\documentclass[12pt,a4paper]{amsart}
\pdfoutput=1

\usepackage{amsmath,amssymb,amsthm}  
\usepackage{mathtools}
\usepackage{tikz}
\usepackage{thmtools, thm-restate}%For duplicating theorem numbers.

\usepackage{xcolor} 	
\usepackage{hyperref}
\hypersetup{
	colorlinks,
    linkcolor={red!60!black},
    citecolor={green!60!black},
    urlcolor={blue!60!black},
}
\usepackage[abbrev, msc-links]{amsrefs} 
\usepackage{comment}

\usepackage[utf8]{inputenc}
\usepackage[T1]{fontenc}
\usepackage{lmodern}
\usepackage[babel]{microtype}
\usepackage[english]{babel}

\usepackage{geometry}
\usepackage{enumitem}

\theoremstyle{plain}
\newtheorem{thm}{Theorem}[section]

\newtheorem{prop}[thm]{Proposition}

\newtheorem{claim}[thm]{Claim}
\newtheorem{cor}[thm]{Corollary}
\newtheorem{lem}[thm]{Lemma}

\newtheorem{obs}[thm]{Observation}

\newtheorem{conj}[thm]{Conjecture}

\theoremstyle{definition}
\newtheorem{rem}[thm]{Remark}

\title{Proof of Nash-Williams' Intersection Conjecture for countable matroids}

\author{Attila Jo\'{o}}
\thanks{The author would like to thank the generous support of the Alexander 
von Humboldt Foundation and NKFIH 
OTKA-129211}
\address{Attila Jo\'{o},
University of Hamburg, Department of Mathematics, Bundesstra{\ss}e 55 (Geomatikum), 20146 Hamburg, Germany}
\email{attila.joo@uni-hamburg.de}
\address{Attila Jo\'{o},
Alfr\'{e}d R\'{e}nyi Institute of Mathematics, Set theory and general topology research division, 13-15 Re\'{a}ltanoda St., 
Budapest, Hungary}
\email{jooattila@renyi.hu}


\begin{document}
\begin{abstract}
    We prove that if $ M $ and $ N $ are  finitary  matroids on a  common countable edge  set  $ E $ then they admit a common 
    independent set $I $ such that there is a bipartition $ E=E_{M}\cup E_{N} $ for which $ I\cap E_M  $ spans $ E_M $ in $ 
    M $ and $ I\cap E_N  $ spans $ E_N $ in $ N $.  It answers positively the  Matroid Intersection Conjecture of Nash-Williams 
    in the countable case. 
\end{abstract}

\maketitle

\section{Introduction}

The Matroid Intersection Conjecture of Nash-Williams \cite{MR1678148} has been one of the most important open problems in 
infinite matroid 
theory for decades. It contains as a special case the generalization of Menger's theorem to infinite graphs conjectured by 
Erd\H{o}s and proved by Aharoni and Berger (see \cite{aharoni2009menger} and \cite{MR3784779}). The Matroid 
Intersection 
Conjecture is a 
generalization of the Matroid Intersection Theorem of Edmonds \cite{edmonds2003submodular} to infinite matroids based on 
the complementary 
slackness conditions  (cardinality is usually an overly rough measure to obtain deep results in infinite 
combinatorics).

\begin{conj}[Matroid Intersection Conjecture by Nash-Williams, \cite{MR1678148}]\label{MIC}
If $ M $ and $ N $ are (potentially infinite) matroids on the same edge set $ E $, then they admit a common 
 independent set $ I $ for which there is a bipartition $ E=E_M\cup E_N $ such that $ I_M:=I\cap E_M $ spans $ E_M $ in $ 
 M $ and  $ I_N:=I\cap E_N $ spans $ E_N $ in $ N $.  
\end{conj}
A ``potentially infinite matroid'' originally meant 
 an $ (E, \mathcal{I}) $ with $ \mathcal{I}\subseteq \mathcal{P}(E) $ where:
\begin{enumerate}
 [label=(\roman*)]
\item\label{item fin 1} $ \varnothing\in \mathcal{I} $;
\item\label{item fin 2} $ \mathcal{I} $ is downward closed;
\item\label{item fin 3} For every finite $ I,J\in \mathcal{I} $ with $ \left|J\right|=\left|I\right|+1 $, there exists an 
$  e\in 
J\setminus I $ such 
that
$ I+e\in \mathcal{I} $;
\item\label{item finitary} If every  finite subset of an $ X\subseteq E $ is in $ \mathcal{I} $, then $ X\in \mathcal{I} $.
\end{enumerate}

Matroids satisfying the axioms above are called nowadays \emph{finitary} and they form a proper subclass of  matroids.
Adapting the terminology introduced by Bowler and Carmesin  in \cite{MR3347465}, we say that 
the matroid pair $ \{ M,N \} $ 
(where $ E(M)=E(N) $) has the \emph{Intersection property} if they admit a common independent set
demanded by  Conjecture \ref{MIC}. The first (and for a long time the only) partial result on Conjecture 
\ref{MIC} was due 
to Aharoni and Ziv:

\begin{thm}[Aharoni and Ziv, \cite{MR1678148}]\label{thm Aharoni Ziv}
Let $ M $ and $ N $ be  finitary matroids on the same countable edge set and assume that  $ M $ is the direct sum of 
matroids of finite rank. Then $ \{ M,N \} $ has the Intersection property.
\end{thm}
Our main result is to omit completely the extra assumption about $ M $:

\begin{thm}\label{thm matr int 0}
Let $ M $ and $ N $ be  finitary matroids on the same countable edge set. Then $ \{ M,N \} $ has the Intersection property.
\end{thm}
Finitary matroids were not considered as an entirely satisfying  infinite generalisation of matroids because they fail 
to capture a key phenomenon of the finite theory, the duality. 
Indeed, the class of  finitary matroids is not closed under taking duals, namely the set of subsets of $ E $ avoiding some 
$ \subseteq $-maximal element of $ \mathcal{I} $ 
does not necessarily satisfy axiom \ref{item finitary}.  Rado asked in 1966 if there is a reasonable notion of infinite matroids 
admitting 
duality and minors. Among other attempts Higgs introduced \cite{MR274315} a class of structures he called ``B-matroids''. 
Oxley gave an axiomatization of B-matroids and showed that
 they are the largest class of structures satisfying axioms \ref{item fin 1}-\ref{item fin 3} and  closed under taking
duals and minors  (see \cite{MR1165540} and \cite{MR540005}). Despite these discoveries of Higgs and Oxley, the 
systematic investigation of infinite matroids started only around 2010 when Bruhn, Diestel, Kriesell, Pendavingh, Wollan  
found a 
set of cryptomorphic axioms for infinite matroids, generalising the usual independent set-, bases-, circuit-, closure- and 
rank-axioms for finite mastoids and showed that several well-known facts of the theory of finite matroids are preserved (see 
\cite{MR3045140}). 

An $ M=(E, \mathcal{I}) $ is a B-matroid (or simply matroid) if $ \mathcal{I}\subseteq \mathcal{P}(E) $ with
\begin{enumerate}
[label=(\arabic*)]
\item\label{item axiom1} $ \varnothing\in  \mathcal{I} $;
\item\label{item axiom2} $ \mathcal{I} $ is downward closed;
\item\label{item axiom3} For every $ I,J\in \mathcal{I} $ where  $J $ is $ \subseteq $-maximal in $ \mathcal{I} $ but $ I $ is 
not, there 
exists an
 $  e\in 
J\setminus I $ such that
$ I+e\in \mathcal{I} $;
\item\label{item axiom4} For every $ X\subseteq E $, any $ I\in \mathcal{I}\cap 
\mathcal{P}(X)  $ can be extended to a $ \subseteq $-maximal element of 
$ \mathcal{I}\cap \mathcal{P}(X) $.
\end{enumerate}

This more general matroid concept  gave a broader interpretation for Conjecture \ref{MIC}. Bowler and Carmesin showed in 
\cite{MR3347465} that several important conjectures  in infinite matroid theory are equivalent with Conjecture 
\ref{MIC} and gave a simpler proof for a slightly more general form of Theorem \ref{thm Aharoni Ziv}.

To state  a more general form of our main result Theorem \ref{thm matr int 0}  together with a couple of side results, we recall 
some notions. 
A matroid is called \emph{finitary} if all of its circuits are finite (equivalently: if satisfies \ref{item finitary}). The 
\emph{finitarization} of  $ M $ is a matroid on the same edge set
 whose circuits are exactly the finite circuits of $ M $.  A 
matroid $ M $ is \emph{nearly finitary} if every base of $ M $ 
can be extended to a base of its finitarization by adding finitely many edges.  A matroid is \emph{(nearly) cofinitary} if its 
dual is (nearly) finitary. The \emph{cofinitarization} of $ M $ is the dual of the finitarization of  $ M^* $.
\begin{thm}\label{thm matr int}
If $ M $ and $ N $ are matroids on a common countable edge set  where each of them is either nearly finitary or nearly 
cofinitary, then $ \{ M,N \} $ has the 
Intersection property.
\end{thm}

For matroids $ M $ and $ N $ on the same edge set $ E $,  $ \boldsymbol{\mathsf{cond}(M,N)} $ stands for  the condition ``for 
every $ 
W\subseteq E $ for which there is a base of
$ M\upharpoonright W $
independent in  $ N.W $,   there exists a base of $ N.W $ which is independent  in  $ M\upharpoonright W $''. The next 
theorem says that $ \mathsf{cond}(M,N) $ is a necessary 
and sufficient condition for the existence of a set which is  independent in $ M $ and spanning in  $ N $. 

\begin{thm}\label{thm ind span}
Let $ M $ and $ N $ be matroids on a common countable edge 
set  such that each of them is either nearly finitary or nearly cofinitary. Then there is a base of $ N $ which is 
independent 
in $ M 
$ if and only if  $ 
\mathsf{cond}(M,N) $ holds.
\end{thm}

Looking for an $ M $-independent $ N $-base can be rephrased as searching  for an $ N $-base contained in an $ M $-base. It 
seems natural to ask 
about a characterisation for having a common base.
\begin{thm}\label{thm common base}
Let $ M $ and $ N $ be   matroids on a common countable edge set such that each of them is either  finitary or  cofinitary. Then 
$ M$ and $ N $ have a common 
base  if and only if $ \mathsf{cond}(M,N)\wedge  \mathsf{cond}(N,M) $ holds.
\end{thm}

Maybe surprisingly, the generalization of Theorem \ref{thm common base} for arbitrary countable matroids is consistently 
false
(take $ U $ and $ U^{*} $ from Theorem 5.1 of \cite{erde2019base}).
In  contrast to our other results, we do not even know if ``finitary or  cofinitary'' can be relaxed to ``nearly finitary or  
nearly 
cofinitary'' in Theorem 
\ref{thm common base}.

It is worth  mentioning that if $ M $ is nearly finitary and $ N $ is nearly cofinitary with $ E(M)=E(N)  $, then  $ \{ M, N \} $ 
has the Intersection 
property (regardless of the size $ \left|E(M)\right| $):

\begin{thm}[Aigner-Horev, Carmesin and Frölich; Theorem 1.5 in \cite{MR3784779}]\label{mixed}
If $ M $ is a nearly 
finitary and $ N $ is a nearly cofinitary matroid 
on a common edge set, then $ \{ M,N \} $ has the Intersection property.
\end{thm} 

\iffalse
Indeed, the finiteness of $ M $-circuits and $ N $-cocircuits make possible to 
find by compactness  an $ M $-independent  $ I $ for which $ 
\mathsf{span}_N(I) $ is maximal among the spans of $ M $-independent sets. Clearly $ I $ can be chosen to be $ N 
$-independent as well 
by ``trimming'' what we get. Then this $ I $ cannot be improved by augmenting paths (because its maximizing property) which is 
known to be equivalent to the fact that $ I $ witnessing the Intersection property for $ \{ M,N \} $. 
\fi

The paper is structured as follows. After  introducing a few notation  in the next section we recall the augmenting path 
method in Edmonds' proof of the Matroid Intersection Theorem in Section \ref{sec augpath}  and analyse the changes of the 
auxiliary digraph 
after an augmentation. In Section \ref{sec waves} we remind the so called ``wave''  technique developed by Aharoni and 
prove some 
properties of waves.
 We show in Section \ref{sec reductions} that the restriction of Theorem \ref{thm ind span} to finitary matroids implies all 
 the theorems we are 
intended to prove 
and  from that point we focus only on this theorem. In Section \ref{sec feasible sets} we investigate feasible sets, i.e., 
common independent sets  
 $ I $ of $ M $ and $ N $
satisfying $ \mathsf{cond}(M/I,N/I) $. The intended meaning of  ``feasible''  is being extendable to an $ M $-independent 
base of $ N $. The main result is proved in Section \ref{sec proof of main res} and its core is  
Lemma \ref{key lemma} which enables us to find a feasible extension of a given feasible set which spans in $ N $ a 
prescribed edge.
\section{Notation and basic facts}\label{sec notation}

In this section we introduce some notation and recall some basic facts about 
matroids that  we will use later without further explanation. For more details we refer to  \cite{nathanhabil}.

A pair ${M=(E,\mathcal{I})}$ is a \emph{matroid} if ${\mathcal{I} \subseteq \mathcal{P}(E)}$ satisfies  the axioms 
\ref{item axiom1}-\ref{item axiom4}.
The sets in~$\mathcal{I}$ are called \emph{independent} while the sets in ${\mathcal{P}(E) \setminus \mathcal{I}}$ are 
\emph{dependent}. An $ e\in E $ is a \emph{loop} if $ \{ e \} $ is dependent.
If~$E$ is finite, then  \ref{item axiom1}-\ref{item axiom3}  are equivalent to the usual axiomisation of matroids in terms of 
independent sets  (while \ref{item axiom4} is redundant).
The maximal independent sets are called \emph{bases} and the minimal dependent sets referred to as \emph{circuits}. If $ M $ 
admits a finite base, then all the bases have the same size which 
is the rank $ \boldsymbol{r(M)} $ of $ M $ otherwise we let $ r(M):=\infty $. 
Every dependent set contains a circuit (which fact is not obvious if $ E $ is infinite). If $ C_1,C_2 $ are circuits with $ e\in 
C_1\setminus C_2 $ and $f\in  C_1\cap C_2 $, then there is a circuit $ C_3 $ with $e\in C_3\subseteq C_1\cup C_2-f $. We 
say that $ C_3 $ is obtained by \emph{strong circuit elimination} from $ C_1 $ and $ C_2 $ keeping $ e $ and removing $ f 
$. The 
\emph{dual} of a 
matroid~${M}$ is the 
matroid~${M^*}$ with $ E(M^*)=E(M) $ whose bases are the complements of 
the bases of~$M$. For an  ${X \subseteq E}$, ${\boldsymbol{M  \upharpoonright X} :=(X,\mathcal{I} 
\cap \mathcal{P}(X))}$ is a matroid and it  is called the \emph{restriction} of~$M$ to~$X$.
We write ${\boldsymbol{M - X}}$ for $ M  \upharpoonright (E\setminus X) $  and call it the minor obtained by the 
\emph{deletion} of~$X$. 
The \emph{contraction} of $ X $ in $ M $ and the contraction of $ M $ onto $ X $ are 
${\boldsymbol{M/X}:=(M^* - X)^*}$ and $\boldsymbol{M.X}:= M/(E\setminus X) $ respectively.  
Contraction and deletion commute, i.e., for 
disjoint 
$ X,Y\subseteq E $, we have $ (M/X)-Y=(M-Y)/X $. Matroids of this form are the  \emph{minors} of~$M$. 
If $ I $ is independent in $ M $  but $ I+e $ is dependent for some $ e\in E\setminus I $  then there is a unique 
circuit   $ \boldsymbol{C_M(e,I)} $ of $ M $ through $ e $ contained in $ I+e $.  We say~${X 
\subseteq E}$ \emph{spans}~${e \in E}$ in matroid~$M$ if either~${e \in X}$ or there exists a circuit~${C 
\ni e}$ with~${C-e \subseteq X}$. 
We denote the set of edges spanned by~$X$ in~$M$ by~$\boldsymbol{\mathsf{span}_{M}(X)}$. 
An ${S \subseteq E}$ is \emph{spanning} in~$M$ if~${\mathsf{span}_{M}(S) = E}$.

\section{Augmenting paths}\label{sec augpath}
The Matroid Intersection Theorem states (using our terminology) that every pair of  matroids on the same finite edge set  
has 
the Intersection 
property. It is a fundamental tool in combinatorial optimization and has a great importance since it has been discovered by 
Edmonds \cite{edmonds2003submodular}. The polynomial algorithm in Edmonds' 
proof finds a maximal sized common 
independent set together with a bipartition  witnessing  optimality. It improves a common independent set iteratively via 
augmenting 
paths taken in an auxiliary digraph. 

In the infinite case these augmenting paths are working in the same way and will play an important 
role in our proof. 
However, they are insufficient alone to prove our main result. Indeed, applying augmenting paths recursively yields a sequence 
of common independent sets where a reasonable limit object cannot be guaranteed in general. In this 
subsection we introduce our 
terminology about augmenting paths and prove some properties which were irrelevant for Edmonds' proof but  are 
crucial for our arguments.

Let $ M $ and $ N $ be fixed arbitrary matroids on the same edge set $ E $. For a common independent set $ I $, let 
$ \boldsymbol{D(I)} $ be a digraph on $ E $ with the following arcs. For  $ e\in I $ and $ f\in E\setminus I $,  
$ ef\in D(I) $  if   $ f\in \mathsf{span}_N(I) $ with $ e\in C_N(f,I) $ and $ fe\in D(I) $   if   $ f\in \mathsf{span}_M(I) 
$ with 
$ e\in C_M(f,I) $.   An \emph{augmenting path}
for $ I $ is a $ \subseteq $-minimal  $ P\subseteq E $ of odd size admitting a linear ordering
$ P=\{ x_0,\dots,x_{2n} \}$, for which
\begin{enumerate}
\item $ x_0\in E\setminus \mathsf{span}_N(I)$,
\item $ x_{2n}\in E\setminus \mathsf{span}_M(I) $,
\item $ x_kx_{k+1} \in D(I) $ for $ k<2n $.
\end{enumerate}
Observe that  each $ x_k $ with $ 0<k<2n $ is spanned by $ I $ in both matroids. Furthermore,  by the minimality of $ P $ there 
cannot be 
 $ k+1<\ell  $ with $ x_kx_{\ell}\in D(I) $ (i.e., there are no `jumping arcs'). Therefore the linear order witnessing that $ P $ 
 is 
 an augmenting path for 
$ I $ is unique. If there is no augmenting path 
for 
$ I $, then 
the set 
$ E_M $ of elements reachable from
$E\setminus \mathsf{span}_N(I) $ in $ D(I) $ together with $ E_N:=E\setminus E_M $ 
witnessing the Intersection property of $ \{ N,M \} $.

Let an augmenting path $ P=\{ x_0,\dots,x_{2n} \}$ for $ I $ be fixed.

\begin{lem}\label{arc remain lemma}
If  $ P $ contains neither  $e$ nor any of its out-neighbours with respect to $ D(I) $, then 
$ ef \in D(I \vartriangle P) $ whenever $ef\in  D(I) $.
\end{lem}

\begin{proof}
For an $ e\in E\setminus I $, its out-neighbours are $ C_M(e,I)-e $ (or $ \varnothing $ if $ e\notin \mathsf{span}_M(I) $ which 
case is irrelevant). By 
assumption $ 
P\cap C_M(e,I)=\varnothing $ and 
therefore $C_M(e,I)-e\subseteq  I \vartriangle P$ thus $ C_M(e,I)=C_M(e,I \vartriangle P) $. This means by definition that $ e $ 
has the same out-neighbours in 
$ D(I) $ and $ D(I \vartriangle P) $.

Assume now that $ e\in I $ and $ ef\in D(I) $ (i.e., $ e\in C_N(f,I) $)  for some $ f $.
For $ k\leq n $, let us denote 
$ I+x_0-x_1+x_2-\hdots -x_{2k-1}+x_{2k} $ by $ I_k $. Observe that
$ {I_n=I \vartriangle P} $. We show by induction on $ k $ that $ I_k $ is $ N $-independent and
$ {e\in C_N(f, I_k)} $. Since 
$ I+x_0 $ is $ N $-independent by definition and $ x_0\neq f $ by assumption, we obtain $ C_N(f,I)=C_N(f,I_0) $.  Suppose 
that 
we already know the statement for some  $ k<n $. We have $ C_N(x_{2k+2},I_k)=C_N(x_{2k+2},I)\ni x_{2k+1} $ because 
there is no jumping arc in the augmenting path. It follows that  $ I_{k+1} $ is $ N $-independent.
If $ x_{2k+1}\notin C_N(f, I_k) $ then $ C_N(f, I_k)=C_N(f, I_{k+1}) $ 
and the induction step is done. Suppose that   $ x_{2k+1}\in C_N(f, I_k) $. Note that $e\notin C_N(x_{2k+2},I) $ since 
otherwise $P$ would contain the out-neighbour $x_{2k+2} $ of $ e $ in $ D(I) $. We apply strong circuit elimination 
with  $ C_N(f, I_k) $ and 
$ C_N(x_{2k+2},I_k) $ keeping $ e  $ and removing $ x_{2k+1} $. The resulting circuit $ C\ni e $ can have
at most one element out of $ I_{k+1} $, namely $ f $. Since $ I_{k+1} $ is $ N $-independent, there must be  at least one 
such 
an element and therefore $ C=C_N(f,I_{k+1}) $.
\end{proof}

\begin{cor}\label{aug what span}
$\mathsf{span}_N(I \vartriangle P) =\mathsf{span}_N(I+x_0)$ and
$  \mathsf{span}_M(I \vartriangle P) =\mathsf{span}_M(I+x_{2n}) $.
\end{cor}
\begin{proof}
In the proof of Lemma \ref{arc remain lemma}, $ I_{k+1} $ is obtained 
from $ I_k $ by replacing  $ x_{2k+1}\in I_k $ by $ x_{2k+2} $ for which
$ x_{2k+1}\in C_N(x_{2k+2}, I_k) $ thus $ \mathsf{span}_{N}(I_k)=\mathsf{span}_{N}(I_{k+1})  $. Since
$ I_0=I+x_0 $ and $ I_n=I \vartriangle P $, $\mathsf{span}_N(I \vartriangle P) =\mathsf{span}_N(I+x_0)$ follows. The proof 
of the second part goes similarly.
\end{proof}

\begin{obs}\label{arc remain fact}
If $ ef \in D(I) $ and $ J\supseteq I $ is a common independent set of $ N $ and $ M $ with 
$ \{ e,f \}\cap J=\{ e,f \}\cap I $, then $ ef \in D(J) $ (the same circuit is the witness).
\end{obs}

\section{Waves}\label{sec waves}

Waves were introduced by Aharoni to solve matching problems in infinite bipartite graphs. These techniques turned out to be 
useful in the 
proof of the Erdős-Menger Conjecture by Aharoni and Berger \cite{aharoni2009menger} and in the already mentioned result 
\cite{MR1678148} about the 
Matroid Intersection Conjecture. Let $ M $ and $ N $ be arbitrary matroids on the same edge set $ E $.
An $ (M,N) $-\emph{wave} is a $ W\subseteq E $ such that there is a base of $ M\upharpoonright W $ which is independent in  
$ N.W $. If $ (M,N) $ is clear from the context  we write simply wave. A set $ L $ of $ M $-loops is a wave witnessed by 
$ \varnothing $. We call such a wave \emph{trivial}.

\begin{prop}\label{wave union}
The union of arbitrary many waves is a wave.
\end{prop}
\begin{proof}
Suppose that $ W_{\beta} $ is a wave  for $ \beta<\kappa $ and let 
$W_{<\alpha}:=\bigcup_{\beta<\alpha}W_{\beta}  $ for $ \alpha\leq \kappa $. We  fix a base 
$ B_{\beta} \subseteq W_{\beta} $
of $ M\upharpoonright W_{\beta} $ which is independent in $ N.W_{\beta} $.  Let us define 
$ B_{<\alpha} $  by transfinite recursion for $ \alpha\leq \kappa $ as follows.

\[B_{<\alpha}:= \begin{cases} \varnothing &\mbox{if } \alpha=0 \\
B_{<\beta}\cup (B_\beta \setminus W_{<\beta}) & \mbox{if } \alpha=\beta+1\\
\bigcup_{\beta<\alpha}B_{<\beta} & \mbox{if } \alpha \text{ is limit ordinal}. 
\end{cases} \]

 First we show  by transfinite induction
 that $ B_{<\alpha} $ is spanning in $ M\upharpoonright W_{<\alpha} $. For $ \alpha=0 $ it is trivial. 
 For a limit $ \alpha $ it follows directly from the induction hypothesis. If $ \alpha=\beta+1 $, then
 by the choice of $ B_\beta $, the set $ B_\beta \setminus W_{<\beta} $ spans $ W_{<\beta+1}\setminus W_{<\beta} $ 
 in $ M/W_{<\beta} $. Since $ W_{<\beta} $ is spanned by $ B_{<\beta} $ in $ M $ by induction, it follows that 
 $ W_{<\beta+1} $ is spanned by $ B_{<\beta+1} $ in $ M $.
 
 The  independence of $ B_{<\alpha} $ in
 $ N.W_{<\alpha} $   can be reformulated as ``$W_{<\alpha}\setminus B_{<\alpha}$ is spanning in $ 
 N^{*} \upharpoonright 
 W_{<\alpha} $'', 
 which can be proved the same way as above. 
\end{proof}
By Proposition \ref{wave union} there exists a $ \subseteq $-largest $ (M,N) $-wave  that we denote by 
$ \boldsymbol{W(M,N)} $. Note that if $W(M,N) $ is not  witnessing  the violation of $ \mathsf{cond}(M,N) $ (see the 
definition right after Theorem \ref{thm matr int}) then there is 
no such a witness, i.e., $ \mathsf{cond}(M,N) $ holds.

\begin{obs}
If $ W_0 $ is an $ (M,N) $-wave and $ W_1 $ is an $ (M/W_0, N-W_0) $-wave, then $ W_0\cup W_1 $ is an $ (M,N) 
$-wave.
\end{obs}
\begin{cor}\label{empty wave}
For $ W:=W(M,N) $, the largest $ (M/W, N-W) $-wave  is $ \varnothing $.
\end{cor}

\begin{obs}\label{loop span}
If $ \mathsf{cond}(M, N) $ holds and $ L $ consists of $ M $-loops, then $L\subseteq \mathsf{span}_N(E\setminus L) $  
since otherwise wave $ L $ would violate $ \mathsf{cond}(M, N) $).
\end{obs}
\begin{cor}\label{contr base}
Assume that $ \mathsf{cond}(M, N) $ holds, $ X\subseteq E $ and  $ L\subseteq X $ consists of $ M $-loops. Then  any base $ B 
$ of 
$ (N.X)- L  $ is a base of $ N.X $.
\end{cor}
\begin{proof}
For a base $ B' $ of $ N- X $, the set 
$ B\cup B'$ spans $ E\setminus L $ and hence by Observation \ref{loop span} spans the whole $ E $ as well.
\end{proof}

Let us write $ \boldsymbol{\mathsf{cond}^{+}(M, N)} $ for the condition that $ \mathsf{cond}(M, N) $ holds and
 $ W(M,N) $ is trivial (i.e., consists of $ M $-loops).
\begin{lem}\label{one more edge}
Condition $ \mathsf{cond}^{+}(M, N) $ implies that 
 whenever $ W $ is an $ (M/e, N/e) $-wave for some $ e\in E $ witnessed by $ B\subseteq W $, then $ B $ is a common base 
 of $ (M/e) \upharpoonright W $ and $(N/e).W$.
\end{lem}
\begin{proof}
Let $ W $ be an $ (M/e, N/e)  $-wave. Note that 
$(N/e).W=N.W  $ by definition. Pick an $ B\subseteq W $ which is an $ N.W $-independent base of $ (M/e)\upharpoonright 
W $.  We may assume 
that 
$ e\in \mathsf{span}_M(W) $ and $ e $ is not an $ M $-loop. Indeed,  otherwise $  (M/e) \upharpoonright 
W=M\upharpoonright 
W$ holds 
and hence by $ \mathsf{cond}^{+}(M, N) $ we may conclude that $ W $ is trivial and hence $ B=\varnothing $, which is a 
desired 
common base.  

Then $ B $ is not a base in
$ M\upharpoonright W $  but ``almost'', namely $ r((M\upharpoonright W)/B)=1 $.  We apply the 
augmenting path method with $ B$ in regards to  $M\upharpoonright W$ and $  N.W  $. The augmentation 
cannot be successful. Indeed, if $ P $ were an augmenting path then $ B\vartriangle P $ would show that $ W $ is a 
non-trivial $ (M,N) $-wave. Thus we get a bipartition 
$ W=W_{0}\cup W_{1} $ instead where $ B\cap W_0 $ spans $ W_0 $ in $ M $ and $ B\cap W_1 $ spans $ W_1 $ in $ N.W $. 
Observe that $ W_0 $ is an $ (M,N) $-wave and hence it must be trivial and therefore $ B\subseteq W_1 $. By applying 
Corollary \ref{contr base} with $ X:=W$ and  $L:=W_0 $, we may conclude that $ B $ is a base of $ N.W $ (and of $ 
(M/e)\upharpoonright W $ by definition).
 \end{proof}

\section{Reductions}\label{sec reductions}
The first reduction (Corollary \ref{MI ekv ind-span}) will connect Theorems \ref{thm matr int} and \ref{thm ind span}  even 
in a more general form. This was already discovered by Aharoni and Ziv in  \cite{MR1678148}.
\begin{prop}
Let $ M $ and $ N $ be matroids on the common edge set $ E $ such that $ \{ M, N \} $ has the Intersection property. Then
$ \mathsf{cond}(M,N) $ is equivalent with the existence of an $ M $-independent base of $ N $.
\end{prop}
\begin{proof}
The condition $ \mathsf{cond}(M,N) $ is clearly necessary even without any further assumption. To show its sufficiency, let 
$ I=I_M\cup I_N $ and $ 
E=E_M\cup E_N  $ as in Conjecture \ref{MIC}. Then $ I_M $ is an
 $ N.E_M $-independent base of $ M\upharpoonright E_M $ and $ I_N $ is an
  $ M.E_N $-independent base of $ N\upharpoonright E_N $. Therefore $ E_M $ is a wave and by $ \mathsf{cond}(M,N) $ 
  we 
 can pick a $ J $ 
which is a base of $ N.E_M $ and independent in $ M $. Then $B:= I_N\cup J $ is a base of $ N $ and it is also independent 
in $ M $ because $ I_N $ is independent  in $ M.E_N $.
\end{proof}
\begin{obs}\label{closed class}
The matroid classes: finitary, cofinitary, nearly finitary, nearly cofinitary are closed under taking minors. Furthermore, if  
$ \kappa $ is a cardinal and  a class $ \mathcal{C} $ of matroids is closed under taking minors, then so is the subclass 
$ \{ M\in \mathcal{C}: \left|E(M)\right|<\kappa \} $.
\end{obs}
\begin{prop}\label{middle step}
For $ i\in \{ 0,1 \} $ let $ \mathcal{C}_i $ is a class of matroids where $ \mathcal{C}_0 $ is closed under contraction and $ 
\mathcal{C}_1 $ is closed under deletion. Assume that for every 
$ (M,N)\in  \mathcal{C}_0 \times  \mathcal{C}_1 $ with $ E(M)=E(N) $, $ \mathsf{cond}(M,N) $ implies the existence of a 
base 
of $ N $ which is independent in $ M $.
Then every  $ (M,N)\in \mathcal{C}_0 \times  \mathcal{C}_1  $ with $ E(M)=E(N) $ the pair $ \{ M,N \} $ has the Intersection 
property.
\end{prop}
\begin{proof}
Let $ E_M:= W(M,N) $ and let $ I_M $ be a base of $ M\upharpoonright E_M $ which is independent in $ N.E_M $, i.e., $ I_M 
$ is  witnessing that $ 
E_M $ is  a wave. Then $ W(M/E_M, N-E_M)=\varnothing $ by Corollary \ref{empty wave}, in particular 
$ \mathsf{cond}(M/E_M, N-E_M) $ holds. 
Since $ \mathcal{C}_i $ are closed 
under taking minors, we have $ (M/E_M, N-E_M)\in \mathcal{C}_0 \times  \mathcal{C}_1 $ and therefore by assumption we 
can find a base $ I_N $ of 
$ N-E_M $ which is 
independent in $ M/E_M $.
\end{proof}
\begin{cor}\label{MI ekv ind-span}
For $ i\in \{ 0,1 \} $ let $ \mathcal{C}_i $ is a class of matroids where $ \mathcal{C}_0 $ is closed under contraction and $ 
\mathcal{C}_1 $ is closed under deletion. The following are equivalent:
\begin{enumerate}
\item For every $(M,N)\in  \mathcal{C}_0\times \mathcal{C}_1 $ with $ E(M)=E(N) $, $ \{ M,N \} $  has the Intersection 
property.
\item For every $(M,N)\in  \mathcal{C}_0\times \mathcal{C}_1 $ with $ E(M)=E(N) $
 satisfying $ \mathsf{cond}(M,N) $,  there is a base of $ N $ which is independent in $ M $.
\end{enumerate}
\end{cor}

Our next goal is to show that the Matroid Intersection Conjecture 
\ref{MIC} for nearly finitary and nearly cofinitary matroids  can be 
reduced to the case of
finitary and cofinitary ones even if 
the matroids are not countable, more precisely: 
\begin{prop}\label{nearly}
For $ i\in\{ 0,1 \} $, let $ M_i $ be a nearly finitary  or nearly cofinitary  matroid on $ E $  and let $ M_i' $ be its 
finitarization or cofinitarization respectively. 
If  $ \{ M_0', M_1' \} $ has the Intersection property then so does $ \{ M_0, M_1 \} $.
\end{prop}

\begin{proof}
Suppose first that the $ M_i $ are both nearly finitary. Let $ I' $ be a common 
independent set of 
$ M'_0 $ and $ M'_1$ and let $ E=E_0\cup E_1 $ be a bipartition as in Conjecture \ref{MIC}.   By definition, for 
$ I_i':=I'\cap E_i $ we have 
$ r((M_i'\upharpoonright E_i)/I'_i)=0 $. Observe  that 
$ r((M_i\upharpoonright E_i)/X)\leq r((M_i'\upharpoonright E_i)/X) $ for every $ X\subseteq E_i $ because every circuit of 
$ (M_i'\upharpoonright E_i)/X $ is a circuit of $ (M_i\upharpoonright E_i)/X $. From the definition of  `nearly finitary' follows 
directly 
that we can 
delete finitely many elements of $ I' $ to obtain a 
common independent set $ I $ of $ M_0 $ and $ M_1 $. Then for  $I_i:= I\cap E_i  $ we have
$ r((M_i'\upharpoonright E_i)/I_i)<\infty $ and hence by the observation above $ r((M_i\upharpoonright E_i)/I_i)<\infty $ as 
well.

We use the augmenting path method with $ M_0,M_1$ and $ I $. If there is no augmenting path then $ I $ is as desired and we 
are 
done. Otherwise we take an augmenting path $ P $. Since $ P $ has one more elements in $ E\setminus I $ than in $ I $, for 
$ J:=I\vartriangle P $ we have $ \left|J\setminus I\right|= \left|I\setminus J\right|+1<\infty$. 
Thus $\sum_{i=0,1} \left|J_i\setminus I_i\right|= 1+\sum_{i=0,1} \left|I_i\setminus J_i\right|$ where $J_i:= J\cap E_i  $.
Therefore  $\sum_{i=0,1}r((M_i\upharpoonright E_i)/J_i)< \sum_{i=0,1} r((M_i\upharpoonright E_i)/I_i)  $. It 
follows
that after finitely many iterative application of augmenting paths there will be no augmenting path for the resulting common 
independent set $ K $, which means $ K $ witnesses the Intersection property for $ \{ N,N \} $.

If say $ M_0 $ is nearly cofinitary then  the independence in $ M_0' $ implies the independence in 
$ M_0 $.  Although the inequality $ r((M_0\upharpoonright E_0)/X)\leq r((M_0'\upharpoonright E_0)/X) $  for $ X\subseteq E_0 
$ 
does not hold in general, it follows from the definition of  `nearly cofinitary' directly that for 
every $ X\subseteq 
E_0:\  r((M_0'\upharpoonright E_0)/X) <\infty $ implies $  r((M_0\upharpoonright E_0)/X) <\infty $. Based on this implication 
the proof of the nearly finitary case above can be adapted for the nearly cofinitary and mixed cases.
\end{proof}

\begin{obs}[Bowler and Carmesin, \cite{MR3347465}]\label{MIP dual}
If $ \{ M,N \} $ has the Intersection property then so does $ \{ M^{*},N^{*} \} $.
\end{obs}

We will prove in the rest of the paper  the restriction of Theorem \ref{thm ind span} to finitary matroids. All of our results 
follow from it. Indeed,  it implies Theorem \ref{thm matr int} for finitary matroids (see Corollary 
\ref{MI ekv ind-span} and Observation \ref{closed class}). Then the 
generalization to nearly finitary matroids can be obtained by  Proposition \ref{nearly} from which the nearly cofinitary case 
follows by Observation \ref{MIP dual}. The nearly finitary-nearly cofinitary case is solved in Theorem \ref{mixed}. Finally, 
from  Theorem \ref{thm matr int} we 
get 
Theorem 
\ref{thm ind span} by Corollary 
\ref{MI ekv ind-span} from which Theorem \ref{thm common base} follows by applying the following result.

\begin{thm}[Corollary 1.4 of \cite{erde2019base}]\label{thm CB}
Let~$M_i$ be a finitary or cofinitary matroid on the edge set~$E$ for~${i \in \{ 0,1 \}}$. 
If there are bases~$ B_i$, $B_i'$ of~$M_i$ such that ${B_0 \subseteq B_1}$ and ${B_1' \subseteq B_0' }$, 
then~$M_0$ and~$M_1$ share some base. 
\end{thm}

\section{Feasible sets}\label{sec feasible sets}

Let $M $ and $ N $  be some fixed matroids on the same edge set $ E $. An 
$ I\subseteq E $ is  \emph{feasible} (with respect to $ (M,N) $) if $ I $ is a common independent set of $ M $ and 
$ N $ such that $ \mathsf{cond}(M/I, N/I) $ holds. Note that $ \mathsf{cond}(M, N) $ says that $ \varnothing $ is feasible, 
moreover, 
if Theorem \ref{thm ind span} is 
true, then  exactly the feasible sets can be extended to a base of $ N $ which is independent in $ M $. A feasible  $ I $ is 
called \emph{nice} if $ \mathsf{cond}^{+}(M/I, N/I) $ holds (see the definition right before Lemma \ref{one more edge}).

\begin{obs}\label{iterate feasible}
If $ I_0 $ is a common independent set and $ I_1 $ is feasible with respect to $ (M/I_0, N/I_0) $, then $ I_0\cup I_1 $ is 
feasible with respect to $ (M,N) $. If in addition $ I_1 $ is a nice feasible set in regard to $ (M/I_0, N/I_0) $, then so is 
$ I_0\cup I_1 $ to $ (M,N) $.
\end{obs}
\begin{lem}\label{nice extension}
If  $ B $ is  a common base  of $ M\upharpoonright W $ and $ N.W $ for $ W:=W(M,N) 
$, then $ B $  is a nice feasible set. 
\end{lem}
\begin{proof}
 Let $ W':=W(M/B, N/B) $. First we show that 
 $ W'=W\setminus B $ and it consists of $ M/B $-loops. On 
the one hand, $ B $ is spanning in $ M\upharpoonright W $ thus $ W\setminus B $ consists of 
 $ M/B $-loops which gives $ W'\supseteq W\setminus B $. On the other hand, let $ J $ be a witness that $ W' $ is an  $ 
 (M/B, N/B) $-wave. Then $ B\cup J $ ensures that $ W\cup W' $ is an 
 $ (M,N) $-wave. Therefore $ W'\subseteq W $ which yields to $ W'\subseteq W\setminus B $. Thus $ W'=W\setminus B $ 
 consists of $ M/B $-loops as promised. It remains to show that $  \mathsf{cond}(M/B, N/B) $ holds.  From the fact that $ B $ is  
 a base of $ N.W $ 
 we can conclude that $ \varnothing $ is a base of $ N.(W\setminus B) $ which completes the proof.
\end{proof}
\begin{rem}
One may observe that if each of $ M $ and $ N $ are either finitary or cofinitary then $ \mathsf{cond}(M,N) $ implies via 
Theorem \ref{thm CB} 
that for every wave $ W $ there exists a common base $ B $ of $ M\upharpoonright W $ and $ N.W $. For self-readability 
reasons we avoid to use this fact in the proof of the main result.
\end{rem}

\begin{lem}\label{aug nice}
If $ I $ is a nice feasible set and $ P $ is an augmenting path for it, then 
$ I \vartriangle P $ can be extended to a nice feasible set.
\end{lem}
\begin{proof}
It is enough to find a common base $ B $ of $ M/(I \vartriangle P) 
\upharpoonright W  $ and $ (N/(I \vartriangle P)).W $
where $ W:=W(M/(I\vartriangle P), N/(I\vartriangle P)) $. Indeed, by applying  Lemma \ref{nice extension} 
with $ {M/(I\vartriangle P), N/(I\vartriangle P)} $ and $ B $ we may conclude first that $ B $ is nice feasible set with respect to 
$ \left( M/(I\vartriangle P), N/(I\vartriangle P)  \right)   $. Then by using  Observation \ref{iterate feasible}
with $ I\vartriangle P $ and $ B $ we obtain that $ (I\vartriangle P)\cup B $ is a nice feasible  with respect to $ (M,N) $. 

Let $ e $ be the unique 
element 
of $ P\setminus \mathsf{span}_M(I) $.
Corollary \ref{aug what span} ensures that 
$ I+e $ and $ I \vartriangle P $ span each other in $ M $ therefore  $ M/(I+e) \upharpoonright X=  M/(I \vartriangle P) 
\upharpoonright X$ whenever $ X\subseteq E\setminus(I\cup P) $.  For such an $ X $ we
also have $N.X= (N/(I+e)).X=(N/(I \vartriangle P)).X$. In particular the wave subsets of $ E\setminus(I\cup P) $ and the 
associated minors of $ M $ and $ N $ are identical for $ (M/(I\vartriangle P), N/(I\vartriangle P)) $ and for $ (M/(I+e),N/(I+e)) $.  
Let $ W' $ be 
the union of all these common waves.
On the one hand, each $e\in I\cap P $ is a common loop of $  M/(I\vartriangle P)$ and $  
N/(I\vartriangle 
P)$ by Corollary \ref{aug what span}. Hence $ W=W'\cup(I\cap P) $,
furthermore, a common base of $ M/(I \vartriangle P) \upharpoonright W'  $ and $ (N/(I \vartriangle P)).W' $ is automatically 
a common base of $ M/(I \vartriangle P) \upharpoonright W  $ and $ (N/(I \vartriangle P)).W $ as well. On the other hand, by 
applying 
Lemma \ref{one more edge} with $ M/I $ and $ N/I $ and $ e $, there exists a common base  $ B $ of $ M/(I+e) 
\upharpoonright W'  $ 
and $ (N/(I +e)).W' $. This $ B $ is a common base of $ M/(I \vartriangle P) \upharpoonright W'  $ and $ (N/(I \vartriangle P)).W' 
$ since $ W'\subseteq  E\setminus(I\cup P) $.
\end{proof}
\section{The proof of the main result}\label{sec proof of main res}
First we show that we may assume without loss of generality in the proof of Theorem \ref{thm ind span} that $  
\mathsf{cond}^{+}(M,N) $ 
holds. Indeed,  otherwise let us consider $ (M/W,N-W) $ instead of $ (M,N) $ where $ W:=W(M,N) $. 
By Corollary \ref{empty wave},  $ W(M/W,N-W)=\varnothing $, thus in particular $  \mathsf{cond}^{+}(M/W,N-W) $ 
holds. Finally the union of an $ 
M/W $-independent base of $ N-W $ and an $ M $-independent base of $ N.W $ (exists by $  \mathsf{cond}(M,N) $) is a 
desired $ M $-independent base of $ N $.
\begin{lem}\label{key lemma}
If $ M $ and $ N $ are finitary matroids on the common countable edge set $ E $ such that $ \mathsf{cond}^{+}(M,N) $ 
holds, then for every $ e\in E $, there exists a nice feasible $ I $ with ${e\in  \mathsf{span}_N(I)} $.
\end{lem}

Let us fix an enumeration $ \{ e_n: n\in \mathbb{N} \} $ of $ E $ and take a well-order  $ \prec $ on $ E $ according to it. 
Theorem \ref{thm ind span} 
for finitary matroids follows from Lemma \ref{key lemma} by a straightforward recursion.
Indeed, we build an $ \subseteq $-increasing sequence $ (I_n)  $  of nice feasible sets starting 
with $ I_0:=\varnothing $ in such a way that $ {e_n\in  \mathsf{span}_N(I_{n+1})} $. If  $ I_n $ is already defined and $ 
e_n\notin \mathsf{span}_N(I_n) $, then we 
apply Lemma \ref{key lemma} with $ (M/I_n, N/I_n) $ and $ e_n $ and take the union of the resulting $ J $ with  $ I_{n} $ to 
obtain $ I_{n+1} $ (see Observation \ref{iterate feasible}).  Considering that $ M $ and $ N $ are finitary, we may conclude that
 $ \bigcup_{n=0}^{\infty}I_n $ is a base of $ N $ which is independent in $ M $.

 \begin{proof}[proof of Lemma  \ref{key lemma}]
It is enough to build a  sequence $ (I_n) $ of nice feasible sets  such that 
$ \mathsf{span}_N(I_n) $ is  monotone $ \subseteq $-increasing in $ n $ and $ 
\bigcup_{n=0}^{\infty}\mathsf{span}_N(I_n)= E  $. We start with 
$ I_0=\varnothing $ and apply  an augmenting path and add some new edges at each step. Corollary \ref{aug what span} ensures  
that 
$ \mathsf{span}_N(I_n) $ will be  
monotone 
$ \subseteq $-increasing.  Suppose  $ I_n $ is already defined. Assume first that there is no 
augmenting path for $ I_n $. 
 Then there is a bipartition
  $E=E_M\cup E_N  $ witnessing with $ I_n $  the Intersection property of $ \{ M,N \} $. By definition, $ E_M $   is a wave and 
  it must be trivial 
  by $  \mathsf{cond}^{+}(M,N) $. Therefore 
  $ I_n\subseteq  E_N$  and we know that it spans $ E_N $ in $ N $. But then Observation \ref{loop span} (with $ L:=E_M $) 
  ensures that $ I_n $ spans actually the whole $ E $ in $ N $ and therefore $ B:=I_n $ is a 
  base of $ N $ which is 
 independent in $ M $.

 We may assume that there exists some augmenting path for $ I_n $. Let $ P_n $ be an augmenting path with a $ \prec 
 $-smallest 
 possible initial edge with respect to the path order corresponding to $ D(I_n) $. Lemma \ref{aug nice} 
ensures that we can extend $ I_n \vartriangle P_n $ to a nice 
feasible set
$ I_{n+1} $. The recursion is done.

  Suppose for a contradiction that $\boldsymbol{X}:= E\setminus 
\bigcup_{n=0}^{\infty}\mathsf{span}_N(I_n) \neq \varnothing$.
\begin{obs}\label{not jus Mloop}
Since $ N $ is finitary,  Observation \ref{loop span} ensures that there is an edge in $ X $ which is not an $ M $-loop.
\end{obs}

 For $ x\in X $, let $ \boldsymbol{E(x,n)} $ be the set of edges that are 
reachable from $ x $ in $ \boldsymbol{D_n}:=D(I_n) $ by a directed path. Let $\boldsymbol{ n_x} $ be the smallest 
natural number such 
that  for every $ y\in 
E\setminus X $ with   $ y \prec x $ we have $ y\in \mathsf{span}_N(I_{n_x}) $.

\begin{claim}\label{stabilazing stuff}
For every $ x\in X $ and  $  \ell\geq m\geq n_x $, 
\begin{enumerate}
\item\label{item stabilize} $ I_m\cap  E(x,m)=I_{\ell}\cap E(x,m) $,
\item\label{item same circuit}  
$ C_M(e,I_\ell)=C_M(e,I_{m})\subseteq E(x,m)$ for every $e\in  E(x,m)\setminus I_m $,
\item\label{item subdigraph} $ D_m[E(x,m)]$ is a subdigraph of $ D_{\ell}[E(x,m)] $,
\item\label{item increasing}  $ E(x,m)\subseteq E(x,\ell) $.
\end{enumerate}
\end{claim}

\begin{proof}
Suppose that there is an $ n\geq n_x $ such that we know already the statement  whenever  $m,\ell \leq n $. For the induction step 
it is 
enough to 
show that the claim holds for $ n $ and $ n+1 $.
\begin{prop}
 $ P_n \cap E(x,n)=\varnothing $.
\end{prop}
\begin{proof}
A meeting of $ P_n $ and $ E(x,n) $ would show that there is also an augmenting path in $ D_n $ starting at $ x $ which is 
impossible since $ x\in X $ and  $ n\geq n_x $.
\end{proof}
\begin{cor}
$ I_{n}\cap  E(x,n)=(I_n \vartriangle P_n)\cap E(x,n) $.
\end{cor}
\begin{prop}
 $ (I_n \vartriangle P_n)\cap  E(x,n)=I_{n+1}\cap E(x,n) $.
\end{prop}
\begin{proof}
The edges $I_{n+1}\setminus (I_n \vartriangle P_n) $ are independent in $ M/(I_n \vartriangle P_n) $ but by the definition of
$ D_n $, for every 
$e\in  E(x,n)\setminus I_n  $ we 
have $  E(x,n)\supseteq C_M(e,I_n)=C_M(e,I_n \vartriangle P_n)$ witnessing that $ e $ is an $ M/(I_n \vartriangle P_n)$-loop.
\end{proof}
\begin{cor}\label{circ subset}
$ I_n\cap  E(x,n)=I_{n+1}\cap E(x,n) $  and  for every $e\in  E(x,n)\setminus I_n $ we have 
$ C_M(e,I_n)=C_M(e,I_{n+1})\subseteq E(x,n)$.
\end{cor}

Finally,  for $ e\in  E(x,n)$, $ P_n $ does not meet $ e $ or any of its out-neighbours with respect 
to $ D_n $  because $ P_n\cap E(x,n)=\varnothing  $. Hence by applying Lemma \ref{arc remain lemma} with $ e, I_n $  and 
$ D_n $   (and then Observation \ref{arc remain fact}) we may conclude that  $ ef\in D_{n+1} $ whenever   $ef\in  D_n $.  
It follows that 
$ D_n[E(x,n)]$ is a 
subdigraph of $ D_{n+1}[E(x,n)] $  which implies $ E(x,n)\subseteq E(x,n+1) $  since reachability is witnessed by the 
same directed paths.
\end{proof}

Beyond Claim \ref{stabilazing stuff} we will need the following technical  statement.
\begin{prop}\label{circuit simple}
Let $ I $ be an independent set in some fixed  finitary matroid. Suppose that there is a circuit $ C\subseteq \mathsf{span}(I) $ 
with $ e\in I\cap C $.  Then there is an $ f\in C\setminus I $ with $e\in  C(f,I) $.
\end{prop} 

\begin{proof}
We apply induction on $ \left|C\setminus I\right| $. If $ C\setminus I $ is a 
singleton, then its only element is suitable for $ f $ since $ C(f,I)=C  $. Suppose that $ \left|C\setminus I\right| \geq 2 $ and 
pick a $ g\in C\setminus I $. If $ e\in C(g,I)  $, then $ f:=g $ is as desired. Otherwise we 
apply strong circuit 
elimination with $  C$  and $ C(g,I) $ keeping $ e $ and removing $ g $. The resulting $ C'$ satisfies the premisses of the 
proposition and $ C'\setminus I 
\subsetneq C\setminus I $ holds thus we are done by induction.
\end{proof}

Let us recall that according to our indirect assumption  ${ E\setminus 
\bigcup_{n=0}^{\infty}\mathsf{span}_N(I_n)\neq \varnothing} $ and $ X $ was defined to be the left side. In order to get a 
contradiction, we 
show that $ 
{\boldsymbol{W}:=\bigcup_{x\in 
X}\bigcup_{n=n_x}^{\infty} E(x,n)} $ is a non-trivial wave.   Note that property \ref{item stabilize} at Claim \ref{stabilazing 
stuff}  guarantees that for 
each  $ e\in W $ either $ {\{ 
n\in \mathbb{N}: e\in I_n \}} $ or its 
complement is finite.  Let $ J $ consists of the
latter type of edges of $ W$, i.e., that are elements of $ I_n $ for every large enough $ n $. Since $ M $ and $ N $ are 
finitary,  $ J $ is a common independent set. By property \ref{item same circuit},  $W\subseteq 
\mathsf{span}_{M}(J) $. 
 We 
show 
that $ J $ is 
independent in $ N.W $. Suppose for a contradiction that there exists an $ N $-circuit $ C $ that meets $ J $ but 
avoids
$ W\setminus J $.  Since $ J $ is $ N $-independent and $ C $ does not meet $ W\setminus J $, we have 
$C\setminus J= C\setminus W\neq\varnothing  $. Let us pick some $ e\in C\cap J$.
For every large enough $ n $ we have $ C\cap J\subseteq C\cap I_n $ and  $ I_n $ spans  $ C $ in $ N $ (for the latter we use 
$ X\subseteq W\setminus J $). Applying Proposition 
\ref{circuit simple} with $I_n,  N, C $ and $ e $ tells that  $e\in C_N(f,I_n) $ 
for 
some $ f\in C\setminus W $ whenever $ n $ is large enough. 
Then we can take an $ x\in X $ and an $ n\geq n_x $ such that $ e\in E(x,n) \cap C_N(f,I_n) $ 
for some $ f\in C\setminus W $. Then by definition $ f \in E(x,n) \subseteq W $ which contradicts $ f\in C\setminus W $. Thus  
$ J $ is indeed
independent in $ N.W $ and hence  $ W $ is a wave. Observation \ref{not jus Mloop} guarantees that $ W $ is non-trivial  
which contradicts $ \mathsf{cond}^{+}(M,N) $.
\end{proof}

\end{document}